\documentclass[reprint, aps, pre, twocolumn, groupedaddress, showpacs, longbibliography]{revtex4-1}

\usepackage{amsmath}
\usepackage{amssymb}
\usepackage{amsthm}
\usepackage{graphicx}
\usepackage[pdftex,colorlinks]{hyperref}

\newtheorem{thrm}{Theorem}
\newtheorem{lemm}[thrm]{Lemma}
\newtheorem{coro}[thrm]{Corollary}
\theoremstyle{definition}
\newtheorem{deff}[thrm]{Definition}
\newtheorem{nota}[thrm]{Notation}
\newtheorem{rmrk}[thrm]{Remark}
\newtheorem{exmp}[thrm]{Example}

\begin{document}

\title{
Path integral control and state-dependent feedback
}

\author{
	Sep Thijssen
}
\email[]{s.thijssen@donders.ru.nl}
\author{
	H.J. Kappen
}
\email[]{b.kappen@science.ru.nl}
\homepage[]{http://www.snn.ru.nl/~bertk}
\affiliation{%
Department of Neurophysics, \\
Donders Institute for Neuroscience\\
Radboud University Nijmegen, The Netherlands
}%

\date{
	\today
}

\begin{abstract}
In this paper we address the problem of computing state-dependent feedback controls for path integral control
problems. To this end we generalize the path integral control formula and utilize this to construct parametrized
state-dependent feedback controllers. In addition, we show a relation between control and importance sampling:
Better control, in terms of control cost, yields more efficient importance sampling, in terms of effective sample
size. The optimal control provides a zero-variance estimate.
\end{abstract}


\maketitle

\section{Introduction}

Control methods are used widely in many engineering
applications, such as mechanical systems, chemical plants,
finance, and robotics. Often, these methods are used to stabilize
the system around a particular set point or trajectory using state
feedback. In robotics, the problem may be to plan a sequence of
actions that yield a motor behavior such as walking or grasping
an object
\cite{rombokas, kinjo}. 
In finance, the problem may be to devise a
sequence of buy and sell actions to optimize a portfolio of
assets, or to determine the optimal option price \cite{glasserman}.

Optimal control theory provides an elegant mathematical
framework for computing an optimal controller using the
Hamilton-Jacobi-Bellman (HJB) equation. In general the HJB
equation is impossible to solve analytically, and numerical
solutions are intractable due to the problem of dimensionality.
As a result, often a suboptimal linear feedback controller such
as a proportional-integral-derivative (PID) controller \cite{stengel} or
another heuristic approach is used instead. The use of suboptimal controllers may be particularly problematic for nonlinear
stochastic problems, where noise affects the optimality of the
controller.

One way to proceed is to consider the class of control
problems in which the HJB equation can be linearized. Such
problems can be divided into two closely related cases \cite{theodorou2012relative}.
The first considers infinite-time-average cost problems, while
the second considers finite-time problems. Approaches of the
first kind \cite{horowitz, kinjo} solve the control problem as an eigenvalue
problem. This class has the advantage that the solution also
computes a feedback signal, but the disadvantage that a
discrete representation of the state space is required. In the
second case the optimal control solution is given as a path
integral \cite{kappen_prl05}. This case will be the subject of this work.
Path integral approaches have led to efficient computational
methods that have been successfully applied to multiagent
systems and robot movement \cite{rombokas, broek2006, anderson, sugimoto, theu_jmlr}.

Despite its success, two key aspects have not yet been addressed.
\begin{enumerate}
\item 
The issue of state feedback has been largely ignored
in path integral approaches and the resulting ``open-loop''
controllers are independent of the state; they are possibly augmented with an additional PID controller to ensure stability
\cite{rombokas}.
\item
The path integral is computed using Monte Carlo
sampling. The use of an exploring control as a type of
importance sampling has been suggested to improve the
efficiency of the sampling \cite{kappen2005path, glasserman} but there appear to be no
theoretical results to back this up.
\end{enumerate}

These two aspects are related because the exploring controls
are most effective if they are state feedback controls. In
this paper we propose solutions to these two issues. We
generalize the path integral control formula and utilize this to
construct parametrized state-dependent feedback controllers.
In Corollary~\ref{coro:feedback} we show how a feedback controller might be
obtained using path integral control computations that can
be approximated to arbitrary precision in this way if the
parametrization is correct. The parameters for all future times
can be computed using a single set of Monte Carlo samples.

We derive the key property that the path integral is
independent of the importance sampling when using infinite
samples. However, importance sampling strongly affects the
efficiency of the sampler. In Theorem~\ref{thrm:var} we derive a bound
which implies that, when the importance control approaches
the optimal control, the variance in the estimates reduces
to zero and the effective sample size becomes maximal.
This allows us to improve the estimates iteratively by using
better and better importance sampling with increasing effective
sample size.

This work is structured as follows. In Section~\ref{sect:def} we review path integral control and we extend the existing theory in Section~\ref{sect:lemma}. Using this we prove additional variance bounds in Section~\ref{sect:importance}, and generalized path integral control formulas in Section~\ref{sect:main}. In Section~\ref{sect:feedback} we construct a feedback controller, and describe how to compute it efficiently. In Section~\ref{sect:example} we show in an example how to compute several nonlinear feedback controllers for a nonlinear control problem.


\section{The Path Integral Control Problem} 
\label{sect:def}

Consider the dynamical system
\begin{align}\label{eq:sdex}
dX^u(t) &= b(t, X^u(t))dt \nonumber \\
		&\phantom{=\ }+ \sigma(t, X^u(t)) \left[(u(t,X^u(t))dt + dW(t)\right],
\end{align}
for $t_0\leq t\leq t_1$ and with $X^u(t_0) = x_0$. Here $W(t)$ is $m$-dimensional standard Brownian motion, and we take $b:[t_0, t_1]\times \mathbb{R}^n \to \mathbb{R}^n$, $\sigma:[t_0, t_1]\times \mathbb{R}^n \to \mathbb{R}^{n\times m}$ and $u:[t_0, t_1]\times \mathbb{R}^n \to \mathbb{R}^m$ such that a solution of Eq.~(\ref{eq:sdex}) exists. Formulating exact conditions that guarantee existence is not the aim of this work. (See \cite{oksendal, fleming2006controlled} for details of the theory, or \cite{bierkens} for a mathematical approach to path integral control.)

Given a function $u(t, x)$ that defines the control for each state $x$ and each time $t_0\leq t\leq t_1$, we define the cost 
\begin{align}
S^u(t) = 
	&	\int_{t}^{t_1} V(s, X^u(s))+\frac12u(s, X^u(s))'u(s, X^u(s)) ds \nonumber \\
	&+	\int_{t}^{t_1} u(s, X^u(s))'dW(s),\label{eq:s}
\end{align}
where the prime denotes the transpose.
Note that $S$ depends on future values of $X$ and is therefore not adaptive \cite{oksendal, fleming2006controlled} with respect to the Brownian motion. 

It is unusual to include a stochastic integral with respect to Brownian motion in the cost because it vanishes when taking the expectation value. However, when performing importance sampling with $u$, such a term appears naturally (see Section~\ref{sect:importance}). 

The goal in stochastic optimal control is to minimize the expected cost with respect to the control. 
\begin{align*}
&J(t, x) 		= \min_u\mathbb{E}\left[ S^u(t) \mid X^u(t) = x\right], \\
&u^*(\cdot, \cdot)	= \arg \min_u \mathbb{E}[ S^u(t_0) ].
\end{align*}
Here $\mathbb{E}$ denotes the expected value with respect to the stochastic process from Eq.~(\ref{eq:sdex}). The following, previously established result \cite{kappen_prl05, theu_jmlr} gives a solution of the control problem in terms of path integrals. 

\begin{thrm}\label{thrm:old} The solution of the control problem is given by 
\begin{align}
J(t_0, x_0)
	& = -\log \mathbb{E} \ e^{-S^u(t_0)}, 	\label{eq:pi0old}	\\
u^*(t_0, x_0) - u(t_0, x_0) \label{eq:pi1old}
	& = 
		\lim_{t \to t_0} \frac
			{\mathbb{E}\left[e^{-S^u(t_0)} \int_{t_0}^tdW(s) \right]}
			{(t - t_0)\mathbb{E}\left[e^{-S^u(t_0)}\right]}.
\end{align}
\end{thrm}

\begin{proof}
Eq.~(\ref{eq:pi0old}) will be proven in Remark~\ref{rmrk:pi0} and Eq.~(\ref{eq:pi1old}) follows from the generalized Main Theorem in Section~\ref{sect:main}.
\end{proof}

Because the solution of the control problem is given in terms of a path integral Eqs.~(\ref{eq:pi0old}, \ref{eq:pi1old}), the control problem Eqs.~(\ref{eq:sdex}, \ref{eq:s}) is referred to as a path integral control problem. The formulas from Theorem~\ref{thrm:old} provide a solution at $t_0$. Of course, since $t_0$ is arbitrary, this can be utilized at any time $t$. However, for $t > t_0$, the state $X^u(t)$ is probabilistic, and consequently, the optimal control must be recomputed for each $t, x$ separately. This issue will be partly resolved in the Main Theorem, where we show that all expected optimal future controls can be expressed using a single path integral.

The optimal control solution holds for any function $u$. In particular, it holds for $u = 0$ in which case we refer to Eq.~(\ref{eq:sdex}) as the uncontrolled dynamics. Computing the optimal control in Eq.~(\ref{eq:pi1old}) with $u \neq 0$ implements a type of importance sampling, which is further discussed in Section~\ref{sect:importance}.

\begin{rmrk}
It is straightforward, but notationally tedious, to generalize the control problem to the following slightly more general form
\begin{align*}
& dX = bdt + \sigma (udt + \rho dW), \\
& S = \Phi(X^u(T)) + \int_{t_0}^{t_1} V + \frac12 u'Ru dt + \int_{t_0}^{t_1} u'R\rho dW,
\end{align*}
with $\Phi\in\mathbb{R}$, and $R,\sigma\in\mathbb{R}^{m\times m}$ with $\lambda I = R \rho \rho'$ and $\lambda\in\mathbb{R}_{>0}$. Note that we dropped dependence on $t, X^u(t)$ for brevity. 
\end{rmrk}


\section{Linearizable HJB Equation and stochastic processes}
\label{sect:lemma}

In this work we use the HJB equation as a means of solving the control problem. The path integral control problem is characterized by the fact that the HJB equation can be linearized. This will be utilized in this section to obtain the Main Lemma.

\begin{deff}
Throughout the rest of this work we define
\begin{align*}
\psi(t, x) 
	& = e^{- J(t, x)}, \\
\psi(t) 
	& = \psi(t, X^u(t)), \\
\phi(t) 
	& = e^{-S^u(t_0) + S^u(t)} .
\end{align*}
Note that $\psi(\cdot, \cdot)$ denotes a \textit{function} of time and state, while $\phi(\cdot)$ and $\psi(\cdot)$ denote \textit{stochastic processes}, the latter being equal to the function $\psi(\cdot, \cdot)$ of the stochastic process Eq.~(\ref{eq:sdex}). This convention will also be used for other functions, e.g.~$u(t) = u(t, X^u(t))$. We remark that, in contrast to $S^u(t)$, the processes $\psi(t)$ and $\phi(t)$ are adapted: They do not depend on future values of $X$. 
\end{deff}

\begin{lemm} [Main Lemma] \label{lemm:main}
\begin{align}
e^{-S^u(t)} - \psi(t) 
	=	\frac1{\phi(t)}\int_{t}^{t_1}\!\! \phi(s)\psi(s) \left[
			u^*(s) - u(s)
		\right]'\! dW(s).	\label{eq:main} 
\end{align}
\end{lemm}

\begin{proof}
The HJB Equation \cite{fleming2006controlled} for the control problem is 
\begin{align*} 
-J_t
	& =	\min_u \left( 
		V + 
		\frac12 u'u + 
		(b + \sigma u)'J_x + 
		\frac12 \mathrm{Tr}\left(\sigma\sigma' J_{xx}\right)
	\right),
\end{align*}
with boundary condition $J(t_1, x) = 0$. We can solve for $u$ which gives:
\begin{align}
u^*
	& = -\sigma'J_x,	\nonumber \\
-J_t 
	& = V - 
		\frac12 J_x' \sigma \sigma' J_x +
		b' J_x +
		\frac12 \mathrm{Tr}\left(\sigma\sigma' J_{xx}\right). \label{eq:be}
\end{align}
This partial differential equation becomes linear in terms of $\psi$. We have
\begin{align} 
&\psi_t + b'\psi_x + \frac{1}{2}\mathrm{Tr} \ \sigma \sigma' \psi_{xx}=  V\psi,\label{eq:linhjb} \\
&u^*=\frac{1}{\psi} \sigma'\psi_x,	\nonumber 
\end{align}
with boundary condition $\psi(t_1,x) = e^{-J(t_1, x)} = 1$.

Using It\^o's Lemma \cite{oksendal, fleming2006controlled} we obtain a stochastic differential equation (SDE) for the process $\psi(t)$ (dropping the dependence on time for brevity)
\begin{align*}
d\psi
	& = \left( \psi_t + \psi'_x(b + \sigma u) + \frac12 \mathrm{Tr} \ \sigma\sigma' \psi_{xx}\right)dt + \psi'_x\sigma dW \\
	& = V\psi dt + \psi'_x\sigma(udt+ dW),
\end{align*}
where the last equation follows because $\psi(\cdot, \cdot)$ satisfies Eq.~(\ref{eq:linhjb}). From the definition of $\phi$ one readily verifies that it satisfies the SDE $d\phi(t) = -\phi(t)\left(V(t)dt + u(t)'dW(t)\right)$ with initial condition $\phi(t_0) = 1$. Using the product rule from stochastic calculus \cite{oksendal} we obtain 
\begin{align}
d(\phi\psi)
	&=	\psi d\phi + 
		\phi d\psi +
		d[\phi, \psi] \nonumber \\
	&=	-\phi\psi u'dW +
	 	\phi\psi'_x \sigma dW \nonumber \\
	&=	\phi\psi[u^*-u]'dW. \label{eq:mainsde}
\end{align}
Integrating the above from $t$ to $t_1$ gives
\begin{align*}
\phi(t_1)\psi(t_1) &- \phi(t)\psi(t)  \\
	& =	\int_t^{t_1}\phi(s)\psi(s)(u^*(s) - u(s))'dW(s).
\end{align*}
Note that $\psi(t_1) = 1$ and that $\phi(t_1) = \phi(t)e^{-S^u(t)}$. Dividing by $\phi(t)$ we obtain the statement of the lemma. 
\end{proof}


\section{Optimal Importance Sampling} 
\label{sect:importance}

A Monte Carlo approximation of the optimal control solution Eq.~(\ref{eq:pi1old}) is a weighted average, where the weight depends on the path cost. If the variance of the weights is high, then a lot of samples are required to obtain a good estimate. Critically, Eq.~(\ref{eq:pi1old}) holds for all $u$, so that it can be chosen to reduce the variance of the path weights. This induces a change of measure and an importance sampling scheme. By the Girsanov Theorem \cite{oksendal, fleming2006controlled}, the change in measure does not affect the weighted average (for a more detailed description in the context of path integral control, see \cite{theodorou2012relative}). The Radon-Nikodym derivative $e^{-\int (\frac12 u'u dt  + u'dW)}$ is the correction term for importance sampling with $u$, which explains why we included $\int u'dW$ in the definition of $S$. 

In this section we will show that the optimal $u$ for sampling purposes turns out to be $u^*$. More generally, the variance will decrease as $u$ gets closer to $u^*$. This motivates policy iteration, in which increasingly better estimates $u$ of $u^*$ improve sampling so that even better approximations of $u^*$ might be obtained. 

\begin{deff}\label{deff:is} Given the process $X^u(t)$ for $t_0 < t < t_1$:
\begin{enumerate}
\item 
	The weight of a path is defined as $\alpha^u = \frac{e^{-S^u(t_0)}}{\mathbb{E}[e^{-S^u(t_0)}]}$.
\item \label{deff:fes}
	The fraction $\lambda^u$ of effective samples is $\lambda^u = \frac{1}{\mathbb{E}[(\alpha^u)^2]}$.
\end{enumerate}
\end{deff}

\begin{thrm} We have the following upper and lower bounds for the variance of the weight:
\label{thrm:var}
\begin{align}
&\operatorname{Var}(\alpha^u) \leq \int_{t_0}^{t_1}\mathbb{E}\left[(u^* - u)'(u^* - u)(\alpha^u)^2\right]  dt, \label{thrm:varupper} \\
&\operatorname{Var}(\alpha^u) \geq \int_{t_0}^{t_1}\mathbb{E}\left[(u^* - u)\alpha^u\right]'\mathbb{E}\left[(u^* - u)\alpha^u\right] dt. \label{thrm:varlower} 
\end{align}
\end{thrm}

Because $\operatorname{Var}(\alpha^u) + 1 = \mathbb{E}[(\alpha^u)^2]$, the fraction of effective samples as defined in Definition~\ref{deff:is}.\ref{deff:fes} satisfies $0< \lambda^u \leq 1$. It has been suggested \cite{liubook} that this fraction can be used to determine how well one can compute a sample estimate of a weighted average. This can be connected with Theorem~\ref{thrm:var} as follows. 

\begin{coro}\label{coro:fes} If $||u^* - u||^2 \leq \epsilon/(t_1 - t_0)$, then
\begin{align*}
\lambda^u \geq 1 - \epsilon. 
\end{align*}
\end{coro}

\begin{proof}
This follows readily from Eq.~(\ref{thrm:varupper}). 
\end{proof}

A numerical illustration of Theorem~\ref{thrm:var} can be found in Figure~\ref{fig:bound}. Before we prove Theorem~\ref{thrm:var}, we deduce a few useful facts that follow from the Main Lemma.

\begin{coro}
An optimally controlled random path is an instance of Eq.~(\ref{eq:sdex}) with $u = u^*$. Although such a path is random, its attributed cost has zero variance and is equal to the expected optimal cost to go:
\begin{align*}
S^{u^*}(t_0) = - \log \psi(t_0, x_0) = J(t_0, x_0).
\end{align*}
Furthermore we have $\alpha^{u^*} = 1$, such that the weighted average, which is independent of $u$, equals the expectation under the optimal process. 
\end{coro}

\begin{proof}
Take $u = u^*$ and $t = t_0$ in Eq.~(\ref{eq:main}).
\end{proof}

\begin{coro}\label{coro:fk}
The following Feynman-Kac formula \cite{oksendal, fleming2006controlled} expresses $\psi$ as a path integral:
\begin{align}
\psi(t) = \mathbb{E} \left[e^{-S^u(t)} \mid \mathcal{F}_t \right] \label{eq:fk}.
\end{align}
Here the filtration $\mathcal{F}_t$ denotes that we are taking the expected value conditioned on events up to time $t$. 
\end{coro}

\begin{proof}
Take the expected value on both sides of Eq.~(\ref{eq:main}).
\end{proof}

\begin{rmrk}\label{rmrk:pi0}
When we consider Eq.~(\ref{eq:fk}) with $t = t_0$, and take minus the logarithm on both sides, we obtain Eq.~(\ref{eq:pi0old}): a path integral formula for the optimal cost to go function. 
\end{rmrk}

\begin{proof}[Proof of Theorem~\ref{thrm:var}]
Consider Eq.~(\ref{eq:main}) with $t = t_0$, and divide by $\psi(t_0, x_0)$ such that 
\begin{align} 
&\operatorname{Var}\left(\alpha^u\right)  \nonumber \\
	&=	\mathbb{E}\left[ \left(\int_{t_0}^{t_1}  
			\frac{\phi(t)\psi(t)}{\psi(t_0)}[u^*(t) - u(t)]'dW(t)
		\right)^2\right]  
\nonumber \\
	&=	\mathbb{E} \int_{t_0}^{t_1} 
			\frac{\phi(t)^2\psi(t)^2}{\psi(t_0)^2}[u^*(t) - u(t)]'[u^*(t) - u(t)]dt 
			\nonumber \\
	& = \mathbb{E} \int_{t_0}^{t_1} 
			\!\!\left[\alpha^u\psi(t)e^{S^u(t)}\right]^2[u^*\!(t) - u(t)]'[u^*\!(t) - u(t)]dt.
		\label{eq:isom}
\end{align} 	
In the first line we used that $\phi(t_0) = 1$, and in the second line we applied the It\^o Isometry \cite{oksendal}. In the third line we used $\alpha^u = e^{-S^u(t_0)}/\psi(t_0)$, which follows from Eq.~(\ref{eq:fk}) with $t = t_0$. 

For the upper bound we consider Eq.~(\ref{eq:fk}) and apply Jensen's inequality 
\begin{align*}
\psi(t)^2
	= \mathbb{E} \left[
		e^{-S^u(t)}
	\mid \mathcal{F}_t \right]^2
	\leq \mathbb{E} \left[
		e^{-2S^u(t)}
	\mid \mathcal{F}_t \right].
\end{align*}
Substituting in Eq.~(\ref{eq:isom}) and using the Law of total expectation we obtain Ineq.~(\ref{thrm:varupper}).

For the lower bound we use Jensen's Inequality on the whole integrand of Eq.~(\ref{eq:isom}) to obtain
\begin{align*} 		
\operatorname{Var}\left(\alpha^u\right)
	\geq \int_{t_0}^{t_1} 
			\mathbb{E}&\left\{\alpha^u\psi(t)e^{S^u(t)}[u^*(t) - u(t)]'\right\} \\
			&\mathbb{E}\left\{\alpha^u\psi(t)e^{S^u(t)}[u^*(t) - u(t)]\right\}
			dt.
\end{align*} 
Using Eq.~(\ref{eq:fk}) and the Law of total expectation we obtain Ineq.~(\ref{thrm:varlower}).
\end{proof}

We conclude that the optimal control problem is equivalent to the optimal sampling problem. An important consequence, which is given in Corollary~\ref{coro:fes}, is that if the importance control is close to optimal, then so is the sampling efficiency.


\section{The Main Path Integral Theorem}
\label{sect:main}

The Main Theorem is a generalization of Theorem~\ref{thrm:old} that gives a solution of the control problem in terms of path integrals. The disadvantage of Theorem~\ref{thrm:old} is that it requires us to recompute the optimal control for each $t,x$ separately. Here, we show that we can also compute the \textit{expected} optimal future controls using a single set of trajectories with initialization $X(t_0) = x_0$. We furthermore generalize the path integral expressions by considering the product with some function $f(t, x)$. In the next section we utilize this result to construct a feedback controller. Here we proceed with the statement and the proof of the generalized path integral formula.

\begin{nota}
For any process $Y(t)$ we let $\left<Y(t)\right> = \left<Y\right>(t) = \mathbb{E}[\alpha^u Y(t)]$ denote the weighted average.
\end{nota}

\begin{thrm}[Main Theorem] \label{thrm:main}
Let $f:\mathbb{R}\times\mathbb{R}^n\to\mathbb{R}$, and consider the process $f(t) = f(t, X(t))$. Then
\begin{align}
\mathbb{E}[\psi(t)]\label{eq:pi0}
	&= \mathbb{E}\left[e^{-S^u(t)}\right], \\
\left< (u^* - u)f \right>(t) \label{eq:pi1}
	& =\lim_{r\to t}\left< \frac{\int_t^rf(s)dW(s)}{r - t}\right>(t).
\end{align}
\end{thrm}

\begin{proof}[Proof of (\ref{eq:pi0})]
Consider the Feynman-Kac Formula Eq.~(\ref{eq:fk}) and take the expectation with respect to $\mathcal{F}_{t_0}$. 
\end{proof}

\begin{proof}[Proof of (\ref{eq:pi1})]
Consider Lemma~\ref{lemm:main} with $t = t_0$, multiply by $\int_t^rf(s)dW(s)$, and take the expected value:
\begin{align*}
\mathbb{E} & \left[e^{-S^u(t_0)} \int_t^rf(s)dW(s)\right] \\
 & = \mathbb{E} \int_t^r \phi(s)\psi(s)[u^*(s)  -u(s)]f(s) ds.
\end{align*}
On the left-hand side the term $\psi(t_0)\int fdW$ has disappeared because $\psi(t_0)$ is not random and the stochastic integral has zero mean. On the right-hand side we have used independent increments and the It\^o Isometry. Dividing by $r - t$ and taking the limit $r \to t$ we obtain
\begin{align*}
\lim_{r\to t}\frac{1}{r - t}\mathbb{E} &\left[e^{-S^u(t_0)}\int_t^rf(s)dW(s)\right] \\
	& = \mathbb{E} \left[\phi(t) \psi(t) (u^*(t) - u(t))f(t) \right] \\
	& = \mathbb{E} \left[e^{-S^u(t_0)} (u^*(t) - u(t))f(t) \right],
\end{align*}
where in the last line we used that $\phi(t) = e^{-S^u(t_0) + S^u(t)}$ and $\psi(t) = \mathbb{E}[e^{-S^u(t)}|\mathcal{F}_{t}]$ combined with the Law of total expectation. Dividing both sides by $\mathbb{E}[ e^{-S^u(t_0)}]$ gives Eq.~(\ref{eq:pi1}).
\end{proof}


\section{A Parametrized Feedback Controller}
\label{sect:feedback}

In this section we illustrate how Theorem~\ref{thrm:main} can be used to construct a feedback controller. To this end we will assume that $u^*$ is of the following parametrized from:
\begin{align}
u^*(t, x) = A(t)h(t, x). \label{eq:param}
\end{align}
Here $h:\mathbb{R}\times\mathbb{R}^n\to \mathbb{R}^k$ will be referred to as the $k$ `basis' functions which are assumed to be known. The ``parameters'' $A(t)\in\mathbb{R}^{m\times k}$  are assumed to be unknown. Note that the open-loop controller can be obtained by a parametrization with one basis function $h = 1$. The following corollary states that it is possible to estimate the optimal parameters from the equations in the Main Theorem.

\begin{coro}
[Path Integral Feedback] 
\label{coro:feedback}
Let $f(t, x)\in\mathbb{R}^{l}$ be a function, and suppose that $u^*$ is of the form Eq.~(\ref{eq:param}), then 
\begin{align}
A(t)\left< h f' \right>\!(t) = \left< u f' \right>\!(t) + 
	\lim_{r\to t}\left< \frac{\int_t^rf'(s)dW(s)}{r - t}\right>.\label{eq:up}
\end{align}
\end{coro}

\begin{proof}
This follows directly from Eq.~(\ref{eq:pi1}) of the Main Theorem when the parametrized from of $u^*$ is used. 
\end{proof}

Assuming that both the right-hand side and the cross correlations $\left<h f'\right>\!(t)$ can be obtained by sampling methods, Eq.~(\ref{eq:up}) gives for each time $t$ a set of $m\times k$ linear equations in the $k\times m$ unknown parameters $A(t)$. These equations can be solved uniquely if the $k\times l$ matrix $\left<hf'\right>$ is of rank $k$. Although we have to do these computations for each time $t$ separately, only one set of paths is needed to get the sampling estimates for all times. 

In general it will be impossible to check whether the optimal control is of the parametrized form. However, it seems plausible that if the parametrization can represent $u^*$ quite well, it will be possible to estimate a good control function using Corollary~\ref{coro:feedback}. In the next section we perform a numerical experiment to support this statement. 

Note, that we can use any importance control $u$ to estimate the optimal control $u^*$. In principle, we could use $u = 0$ and sample long enough to compute the $u^*$ sufficiently accurately. However, we find it more efficient to use an iterative method where we use the optimal control estimate $u_l$ that was computed at iteration $l$ as an importance control for the computation of the optimal control $u_{l+1}$. According to Corollary~\ref{coro:fes} we know that improved controls have a higher fraction of effective samples and thus will make more efficient use of the sampling data. In particular, if $u$ and $u^*$ are parametrized with the same basis functions and time dependent coefficients $A(t)$ and $A^*(t)$, respectively, this results in an iterative update scheme for these coefficients. We refer to this method as iterative importance sampling. 

We conclude that parametrized control functions can be obtained directly from path integral estimates, where the parameters can be computed using a single set of paths. Critically, these parametrized controls can be state dependent functions. As a result, it is possible to construct (closed-loop) feedback controllers, which are more widely applicable than open-loop controllers. 


\section{Example}
\label{sect:example}

We consider the following control problem, of which we know the analytical solution.

\begin{exmp}[Geometric Brownian Motion] \label{exmp:gbm}
For $t_0 \leq t \leq t_1$, the one-dimensional problem 
\begin{align*}
dX^u(t)
	& =	X^u(t)\left(\frac{dt}2  + u(t, X^u(t))dt + dW(t)\right),	\\
S^u(t)
	& =	\frac{Q}2 \log(X^u(t_1))^2 +
		\frac12\int_{t}^{t_1} u(s, X^u(s))^2ds \\
	& \phantom{=\ } + \int_{t}^{t_1} u(s, X^u(s))'dW(s),
\end{align*}
has solution
\begin{align*}
&u^*(t, x)
	= \frac{-Q\log(x)}{Q(t_1 - t) + 1}.
\end{align*}
For the experiments we will take $x_0 = 1/2$, $t_0 = 0$,  $t_1 = 1$, and $Q = 10$.
\end{exmp}

In a first experiment we visualize Theorem~\ref{thrm:var}. To this end we consider a range of sub-optimal importance controls $u^{\epsilon}(t, x) = u^*(t, x) + \sqrt{\epsilon}$. Each $u^{\epsilon}$ yields a path weight $\alpha^{\epsilon} := \alpha^{u^{\epsilon}}$. Because $\left<u^*- u\right>'\left<u^*- u\right> = \epsilon$, Theorem~\ref{thrm:var} implies that 
$\epsilon \leq\operatorname{Var}(\alpha^{\epsilon})\leq \frac\epsilon{1 - \epsilon}$. The results are reported in Figure~\ref{fig:bound}. 

\begin{figure}
\includegraphics{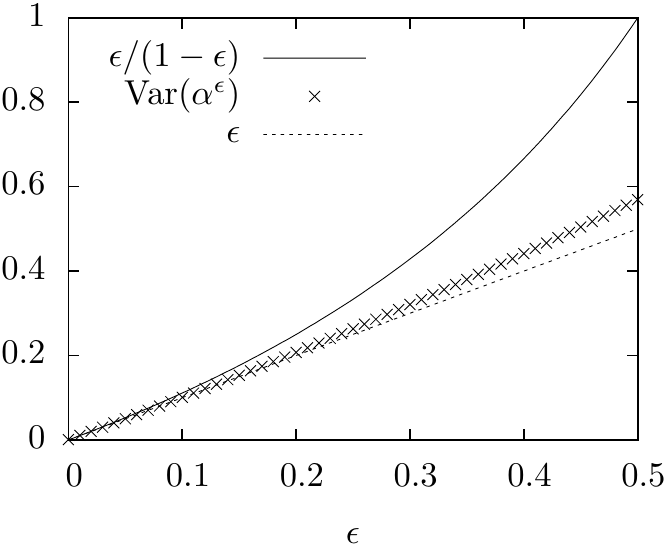}
\caption{
	Estimate of $\operatorname{Var}(\alpha^{\epsilon})$, where $\alpha^{\epsilon} := e^{-S^{u^{\epsilon}}(t_0)}/\psi(t_0, x_0)$ with upper and lower bounds from Theorem~\ref{thrm:var} with respect to the control problem in Example~\ref{exmp:gbm}. Here we considered a range of sub-optimal importance controls $u^{\epsilon}(t, x) = u^*(t, x) + \sqrt{\epsilon}$. The estimate of the variance is based on $10^4$ paths that were generated with $dt = 0.001$. 
}
\label{fig:bound}
\end{figure}

In a second experiment we construct feedback control functions based on various parametrizations. It is clear that a correct parametrization of the problem in Example~\ref{exmp:gbm} can be obtained with just one basis function: $\log(x)$. In the experiment we also consider three parametrizations that cannot describe $u^*$: a constant, an affine, and a quadratic function of the state. The three controllers that we obtain in this way are denoted by $u^{(0)}$, $u^{(1)}$, and $u^{(2)}$, e.g., $u^{(2)}(t, x) = a(t) + b(t)x + c(t)x^2$. 

\begin{table}
\caption{\label{tbl:gbm}
	Performance estimates of various controllers based on $10^4$ sample paths. Although for numerical consistency we used $10^4$ sample paths to compute the parameters, only roughly $10^2$ samples are required to obtain well-performing controllers.
}
\begin{ruledtabular}
\begin{tabular}{l|cccccc}
				& $u = 0$	& $u^{(0)}$	&	$u^{(1)}$ &  $u^{(2)}$ &$a(t)\log(x)$ & $u^*$
\\ \hline 
$\mathbb{E}[S^u(t_0)]$					& 7.526	& 5.139 & 1.507 & 1.461	& 1.422	& 1.420 \\
$\operatorname{Var}(\alpha^u)$	& 1.981	& 1.376	& 0.143	& 0.0506& 0.0085& 0.0071  \\
$\lambda^u$(\%)							& 34.3	& 42.08	& 87.5	& 95.2	& 99.1	& 99.3  
\end{tabular}
\end{ruledtabular}
\end{table}

We have used iterative importance sampling with $f = h$ as described in the previous section to estimate the parameters. The performance of the resulting control functions is given in Table~\ref{tbl:gbm}. The row $\mathbb{E}[S^u(t_0)]$ gives the expected cost, which we want to minimize. The row $\operatorname{Var}(\alpha^u)$ gives the variance of the path weight, which is directly related to the FES. Clearly the open-loop controller $u^{(0)}(t, x) = a(t)$ improves upon the zero controller $u(t, x) = 0$. The control further improves when the affine and quadratic basis functions are subsequently considered. The best result is obtained, unsurprisingly, with the logarithmic parametrization. 

In Figure~\ref{fig:control} we plot the state dependence of the feedback controllers at the intermediate time $t = 1/2$. Although the parametrized functions yield a control for all $x$, we are mainly interested in regions of the state space that are likely to be visited by the process $X$. This is visualized by a histogram of $10^4$ particles that are drawn from $X^{u^*}(1/2)$. We observe that the optimal logarithmic shape is fitted, and that more complex parametrizations yield a better fit. 

\begin{figure}
\includegraphics{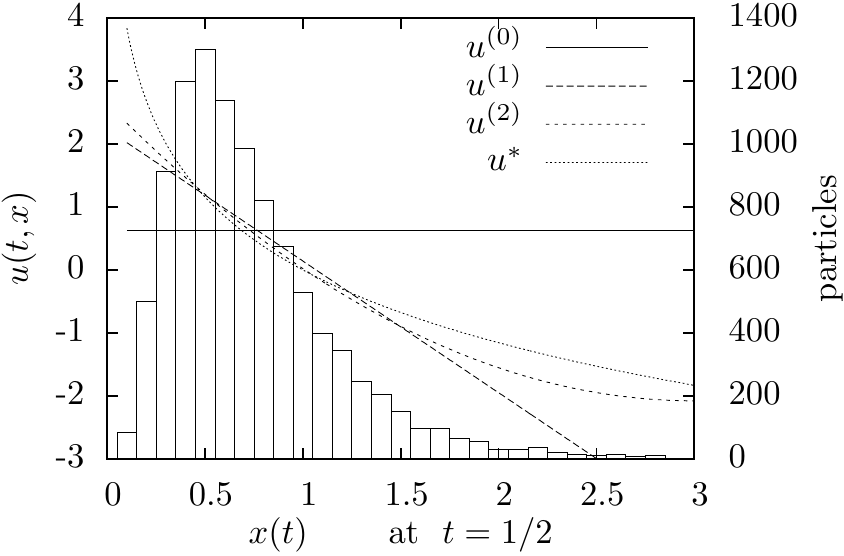}
\caption{
	The approximate controls calculated with $10^4$ sample paths in two importance sampling iterations using a time discretization of $dt = 0.001$ for numerical integration. The histogram was created with $10^4$ draws from $X^{u^*}(t)$ at $t = 1/2$. 
}
\label{fig:control}
\end{figure}


\section{Discussion}

Most current feedback controllers that are used to stabilize systems are linear feedback controllers such as PID controllers. These are heuristic approaches that are optimal only if one assumes that the system dynamics is linear and the cost is quadratic. In this paper we have shown how to compute optimal feedback controllers for a class of nonlinear stochastic control problems. The optimality requires the use of the appropriate basis functions.

It should be noted that the optimal feedback is not necessarily a stabilizing term. Depending on the task it might be optimal to destabilize by amplifying the noise, for example, to create momentum efficiently.

Future work includes the development of methods for practical scenarios, based on the path integral feedback Eq.~(\ref{eq:up}). An important aspect will be the selection of basis functions. A recent related work \cite{horowitz} discusses basis functions to obtain a solution of the linearized HJB Eq.~(\ref{eq:linhjb}). 

\begin{acknowledgments}
This work was funded by D-CIS Lab / Thales Research \& Technology NL, and supported by the European Community Seventh Framework Programme (FP7/2007-2013) under grant agreement 270327 (CompLACS).
\end{acknowledgments}

\bibliography{pi_feedback.bib}

\end{document}